\documentclass[SIADSpaper,final]{siamonline171218}

\usepackage{lipsum}
\usepackage{amsfonts,amssymb,verbatim}
\usepackage{graphicx}
\usepackage{epstopdf}
\usepackage{algorithm,algpseudocode}
\usepackage{mathtools}

\usepackage{geometry}
\usepackage[shadow,prependcaption,textsize=small]{todonotes}
\usepackage{dsfont}
\usepackage{cite}
\ifpdf
  \DeclareGraphicsExtensions{.eps,.pdf,.png,.jpg}
\else
  \DeclareGraphicsExtensions{.eps}
\fi

\graphicspath{{./images/}}

\usepackage[shortlabels]{enumitem}
\setlist[enumerate]{leftmargin=.5in}
\setlist[itemize]{leftmargin=.5in}

\newcommand{\creflastconjunction}{, and~}

\newsiamremark{remark}{Remark}
\newsiamremark{hypothesis}{Hypothesis}
\crefname{hypothesis}{Hypothesis}{Hypotheses}
\newsiamthm{claim}{Claim}
\newsiamthm{example}{Example} %
\newsiamthm{assumptions}{Assumptions} %

\usepackage{amsopn}

\newcommand{\be}{\begin{equation}}
\newcommand{\ee}{\end{equation}}
\newcommand{\beqas}{\begin{eqnarray*}}
\newcommand{\eeqas}{\end{eqnarray*}}

\newcommand{\dee}{\mathrm{d}}

\newcommand{\supp}{\mathrm{supp}}
\newcommand{\intr}{\mathrm{int}}

\def \R {{\mathbb R}}

\def \Q {{\mathbb Q}}

\def \tr {{\mathrm{tr}}}
\def \eps {{\varepsilon}}

\newcommand{\G}{\mathcal{G}}

\numberwithin{theorem}{section}

\newcommand{\TheTitle}{Multiplicative noise in Bayesian inverse problems: Well-posedness and consistency of MAP estimators} 
\newcommand{\TheAuthors}{Matthew. M. Dunlop}

\headers{Multiplicative Noise in Bayesian Inversion}{\TheAuthors}

\title{{\TheTitle}}

\author{
  Matthew M. Dunlop\thanks{Courant Institute of Mathematical Sciences, New York University, New York, New York, 10012, USA (\email{matt.dunlop@nyu.edu}).}
}

\ifpdf
\hypersetup{
  pdftitle={\TheTitle},
  pdfauthor={\TheAuthors}
}
\fi

\graphicspath{{./images/}}

\begin{document}
\maketitle

\begin{abstract}
Multiplicative noise arises in inverse problems when, for example, uncertainty on measurements is proportional to the size of the measurement itself. The likelihood that arises is hence more complicated than that from additive noise. We consider two multiplicative noise models: purely multiplicative noise, and a mixture of multiplicative noise and additive noise. Adopting a Bayesian approach, we provide conditions for the resulting posterior distributions on Banach space to be continuous with respect to perturbations in the data; the inclusion of additive noise in addition to multiplicative noise acts as a form of regularization, allowing for milder conditions on the forward map. Additionally, we show that MAP estimators exist for both the purely multiplicative and mixed noise models when a Gaussian prior is employed, and for the latter prove that they are consistent in the small noise limit when all noise is Gaussian.

\end{abstract}

\begin{keywords}
Bayesian inverse problems, multiplicative noise, MAP estimation, optimization.
\end{keywords}


\section{Introduction}

Let $X$, $Y$ be separable Banach spaces and define a forward operator $\mathcal{G}:X\rightarrow Y$. A common problem is to determine a true state $u \in X$ given observations $y \in Y$ of $\mathcal{G}(u)$. These observations are typically corrupted by noise, which may arise for a multitude of reasons. For example, it may arise from measurement instruments having only finite precision; it may arise from using an inadequate model, in which case it typically depends on the unknown state; or the forward model may be inherently stochastic. Three typical ways to model the noise on $\mathcal{G}(u)$ are as follows:\footnote{Here, juxtaposition of two elements/vectors denotes pointwise/componentwise multiplication respectively.}
\begin{enumerate}[(i)]
\item purely additive noise,
\[
y = \mathcal{G}(u) + \eta_a,\;\;\;\mathbb{E}(\eta_a) = 0;
\]
\item purely multiplicative noise,
\[
y = \eta_m\mathcal{G}(u),\;\;\;\mathbb{E}(\eta_m) = 1;
\]
\item or a mixture of both additive and multiplicative noise,
\begin{align*}
y = \eta_m\mathcal{G}(u) + \eta_a,\;\;\;&\mathbb{E}(\eta_a) = 0,\;\mathbb{E}(\eta_m) = 1,
\end{align*}
\end{enumerate}
where we assume that $\eta_a$ and $\eta_m$ are independent in the latter case.

The additive noise model (i) is the most commonly used noise model, and has a direct relation to least squares optimization, however it is not always the most appropriate model. Examples where the multiplicative noise model (ii) is common include image denoising, wherein noise on each pixel is proportional to the value of the pixel; in this case it is natural for the distribution of the noise to be positive, as pixel values typically have a positive range \cite{goodman1976some}. Other examples include subsurface imaging, where for example measurement errors are proportional to the size of the quantity being measured; here there may be no restriction on sign, and so Gaussian noise may be preferable. The mixed model (iii) may be used either to explicitly model different types of corruption of the data, or simply as a regularized version of model (ii) that leads to greater numerical stability. Note that multiplicative noise could alternatively be viewed as a state-dependent additive noise; such noise can appear when, for example, accounting for model error.

Deterministic methods for solving inverse problems have been extensively studied; see \cite{engl1996regularization} for a review. A common class of methods for solving such problems is variational methods, which involve finding a minimizer of a functional of the form
\[
J(u) =  H(y,u) + \lambda R(u).
\]
Here, $H$ is a misfit function quantifying how well the state agrees with the observed data, and $R$ is a regularization term that promotes certain properties, such as regularity, of the solution. The misfit function is chosen according the data model; the negative log-likelihood of the data $y$ given the state $u$ is typical. In the case of additive noise (i), $H$ is generally a least squares functional, whereas for the other two data models the functional is typically non-convex, even when the forward map $\G$ is linear. In the case of multiplicative noise (ii), when $\G$ is linear, common choices of $H$ include
\begin{align}
\label{eq:gamma}
H(y,u) = \int_D\left(\log{\G(u)} + \frac{y}{\G(u)}\right),
\end{align}
which arises by assuming the noise has a Gamma distribution \cite{Aubert2008,huang2010multiplicative}, and generalizations thereof \cite{ullah2017new,dong2013convex,lu2016multiplicative}.  Other approaches include constrained optimization of the regularization term, such as the RLO method \cite{rudinlions} which minimizes the total variation norm subject to matching the known mean and variance of the observed data. Alternatively, assuming that the noise and forward map are positive, one may transform the data with the logarithm function, turning the multiplicative noise problem into an additive noise one \cite{Achim2006,chen2012multiplicative}; the approaches from case (i) are then available. The mixed noise model appears to have been studied significantly less, however it has been approximated with an additive noise model by learning the scale of the noise from the observations \cite{isaac2015scalable}; we see later that such an approximation may lead to poor estimates in a certain asymptotic regime. In all cases, common choices of regularization for problems with multiplicative noise include total variation and its generalizations \cite{Aubert2008,ullah2017new,chen2012multiplicative,dong2013convex,lu2016multiplicative}, promoting recovery of sharp edges, and Tikhonov \cite{engl1996regularization}, promoting spatial smoothness.

In this article we take a Bayesian approach, so that we obtain a probability distribution on the unknown state, representing uncertainty in the solution, rather than a single state. This distribution is constructed by combining the data likelihood with a prior probability distribution and using Bayes' theorem; in this sense the choice of prior can be viewed as the choice of regularization of the problem. There are two main approaches to studying Bayesian inverse problems: either one discretizes the state space first and applies finite-dimensional methodology for analysis, or one works directly in function space to analyse the problem before discretizing. Both approaches have advantages and disadvantages; and overview of the former may be found in \cite{kaipio2006statistical}, and the latter in \cite{inverse,DS17}. We adopt the latter approach in this article; to the author's knowledge, multiplicative noise models have received very little (if any) explicit study from this approach.

An explicit relation between Bayesian and deterministic approaches lies in finding modes of the posterior distribution. When the state space is finite-dimensional, the relation is clear due to existence of a Lebesgue measure, and in infinite dimensions direct relations have been proved for certain choices of prior \cite{DLSV13,agapiou2018sparsity} once defining the notion of mode appropriately \cite{DLSV13,helin2015maximum,lie2018equivalence}. In principle one may desire to choose a prior with a direct relation to total variation regularization in this manner, given how frequently such a regularization is used in deterministic inversion.
However, the probabilistic analogue is not robust with respect to discretization level \cite{lassas2004can}, and so if edge preservation is desired alternative priors need to be considered in high dimensions. Alternative priors that have recently been developed with edge preservation in mind include Besov \cite{lassas2009discretization,agapiou2018sparsity} and Cauchy \cite{markkanen2019cauchy,sullivan2017well} priors. Such priors leads to some form of sparsity in MAP estimates, analogously to TV regularized optimizers, though they can have other undesirable properties.

In this article we study the posterior distribution arising when the purely multiplicative noise model (ii) and mixed noise model (iii) are employed, proving existence and well-posedness under weak conditions on the choice of prior. Additionally, when the prior is Gaussian, we show existence of MAP estimators, which are characterized as solutions to certain variational problems. We then focus on the case when both the noise and the prior are Gaussian in the mixed noise model, and prove consistency of the MAP estimators in the small noise limit; such results have been shown previously in the case of additive noise \cite{DLSV13}. Though we assume the prior is Gaussian for the latter result, it may be possible to generalize the results to Besov priors analogously to what has been done for problems with additive noise \cite{agapiou2018sparsity}. Throughout we consider only the case where $Y = \mathbb{R}^J$ is finite dimensional; the reasons for this restriction will be discussed later.

\subsection{Our contributions}
\begin{itemize}
\item We formulate the Bayesian inverse problems of recovering the posterior distribution on the unknown state, subject to nonlinear observations perturbed by (a) multiplicative noise and (b) a mixture of multiplicative and additive noise. In both cases, we provide conditions to ensure that the posteriors are continuous with respect to perturbations of the data.
\item We observe that a mixed noise model is more robust than a purely multiplicative noise model, allowing for milder conditions on the forward map.
\item We prove that with mixed Gaussian noise and a Gaussian prior, MAP estimators exist and are consistent in the small noise limit. We also provide motivation for the corresponding result for the large data limit.
\end{itemize}

\section{The Bayesian approach}
We first outline the general Bayesian approach to inversion on infinite-dimensional state spaces, wherein there is no Lebesgue measure to use as a reference. We then justify our present restriction to a finite-dimensional data space $Y = \R^J$.

\subsection{Bayes' theorem and the posterior distribution}
\label{ssec:bayes}
The inverse problems stated in the introduction are typically ill-posed, for example when $X$ has a higher dimension than $Y$. By placing a prior distribution on the state $u$, we can define the solutions of the inverse problems to be the resulting posterior distributions. Well-posedness can then be interpreted to mean that a small change in the data causes only a small change in the posterior distribution, with respect to some metric on measures. 

Assume that our data $y$ arises from any of the three noise models in the introduction. Let $\mathbb{Q}_0$ denote the Lebesgue measure on $Y$, and choose a prior distribution $\mu_0$ on $X$. Define the product measure $\nu_0(\dee u,\dee y) = \mu_0(\dee u)\mathbb{Q}_0(\dee y)$ on $X\times Y$. We define the probability measure $\mathbb{Q}_u$ on $Y$ to be the law of $y|u$ under $\nu_0$. Finally define the probability measure $\nu(\dee u,\dee y) = \mu_0(\dee u)\mathbb{Q}_u(\dee y)$ on $X\times Y$. Assume that $\mathbb{Q}_u \ll \mathbb{Q}_0$ so that there exists $\varphi:X\times Y\rightarrow\mathbb{R}$ with 
\[
\frac{\dee \mathbb{Q}_u}{\dee \mathbb{Q}_0}(y) = \varphi(u,y).
\]
We write $\varphi(u,y) = \exp(-\Phi(u;y) - \Psi(y))$ where $\Phi:X\times Y\rightarrow\mathbb{R}$ and $\Psi:Y\rightarrow\mathbb{R}$, noting that both are only unique up to translation by a function purely dependent on the data. Then since $\nu_0$ is a product measure, we have that
\[
\frac{\dee \nu}{\dee \nu_0}(u,y) = \frac{\dee \mathbb{Q}_u}{\dee \mathbb{Q}_0}(y) = \exp(-\Phi(u;y) - \Psi(y)).
\]

With the above notation in place, we can state the following version of Bayes' theorem, as given in \cite{DS17}.

\begin{theorem}[Bayes' theorem]
\label{thm:bayes}
Assume $\Phi:X\times Y\rightarrow\mathbb{R}$ is $\nu_0$-measurable and that, for $y$ $\mathbb{Q}_0$-a.s.
\[
Z = \int_X \exp(-\Phi(u;y))\,\mu_0(\dee u) > 0.
\]
Then the conditional distribution of $u|y$ exists and is written $\mu^y$. Furthermore, $\mu^y \ll \mu_0$ and, for $y$ $\mathbb{Q}_0$-a.s.,
\[
\frac{\dee \mu^y}{\dee \mu_0}(u) = \frac{1}{Z}\exp(-\Phi(u;y)).
\]
\end{theorem}

Our task is then, given appropriate regularity conditions on the forward map $\mathcal{G}$, the prior $\mu_0$ and the noise, to show that the assumptions of this theorem hold for the different noise models.

\subsection{Data dimensionality and measure singularity}
Above we have made the restriction to finite-dimensional data, $Y = \mathbb{R}^J$. Bayes' theorem may still be formulated similarly to the above when $Y$ is a general separable Banach space, however problems can arise due to issues with measures on such spaces. One such issue is the lack of Lebesgue measure -- instead we would have to take, for example, the dominating measure $\mathbb{Q}_0$ to be Gaussian. However, The Feldman-H\'{a}jek theorem tells us that Gaussians on infinite dimensional spaces are either equivalent or singular: there can be no one-way absolute continuity. By transitivity, we would hence require that $\mathbb{Q}_u\sim\mathbb{Q}_v$ for all $u,v \in X$. With purely additive noise this is easily resolved using the Cameron-Martin theorem: it is sufficient that $\G(u)$ belongs to the Cameron-Martin space of $\Q_0$ for each $u \in X$. With the multiplicative/mixed noise models it is more complicated, since $\Q_u$ is not simply a translation of $\Q_0$.

Though Bayes' theorem can still be formulated abstractly even with measure singularity problems, working directly with the regular conditional probability distributions, the useful property of absolute continuity with respect to the prior distribution is generally lost. As a result, almost-sure properties of the prior, such as sample regularity, do not necessarily pass to the posterior. Another issue related to the measure singularity is when studying the well-posedness of the Bayesian inverse problem: the continuity of the posterior with respect to the data is typically quantified in the Hellinger distance,
\[
d_{\mathrm{Hell}}(\mu,\mu')^2 := \frac{1}{2}\int\left(\frac{\dee \mu}{\dee \nu}(u) - \frac{\dee \mu'}{\dee \nu}(u)\right)^2\,\nu(\dee u).
\]
Here $\nu$ is any probability measure with $\mu,\mu' \ll \nu$. Note that if $\mu$ and $\mu'$ are singular, then $d_{\mathrm{Hell}}(\mu,\mu') = 1$. If an arbitrarily small perturbation in the data leads to singular measures, there can hence be no continuity with respect to the Hellinger distance.

It is possible to define a more general model that incorporates all three of the models described in the introduction; such a model is discussed in the conclusions. For clarity of exposition we restrict here to the explicit multiplicative and mixed noise models as given above -- this allows us to directly compare the assumptions required between the two models. In the future we aim to analyse the general model with infinite-dimensional data.

\section{Existence and well-posedness}
We consider the questions of existence of the posterior measure and sensitivity of the posterior with respect to perturbations of the data, for both the purely multiplicative and mixed noise models. We provide appropriate assumptions for existence and well-posedness of the posteriors in both cases, and observe that mixed noise model allows for weaker assumptions on the forward map. We also verify existence of MAP estimators under these assumptions, the consistency of which will be shown in the following section. Throughout this section, $C$ denotes a generic constant; its dependence on certain parameters is made explicit when relevant.

\subsection{Purely multiplicative noise model}
We first consider the case where the noise on the data is purely multiplicative:
\[
y = \eta_m\mathcal{G}(u)
\]
where $\eta_m$ has Lebesgue density $\rho_m$. Assuming that every component of $\mathcal{G}_j$ of $\mathcal{G}$ is positive and that $\eta_m$ and $u$ are independent, it can be checked that the conditional distribution $y|u$ has density
\[
\rho(y|u) = \frac{1}{\det G(u)}\rho_m(G(u)^{-1}y),
\]
where $G(u) = \mathrm{diag}(\mathcal{G}_j(u)) \in \R^{J\times J}$. The negative log-likelihood $\Phi:X\times Y\rightarrow \mathbb{R}$ is hence given by
\[
\Phi(u;y) = \log\det G(u) - \log\rho_m(G(u)^{-1}y).
\]
If the observations are independent so that $\rho_m$ factors as $\prod_{j=1}^J\rho^j_m$,  this may be written
\[
\Phi(u;y) = \sum_{j=1}^J\left(\log\mathcal{G}_j(u) - \log\rho_m^j\left(\frac{y_j}{\mathcal{G}_j(u)}\right)\right).
\]

\begin{remark}
If we assume in addition that the noise is positive, then the logarithm of the data can be considered to be corrupted by additive noise:
\[
\hat{y} = \log{y} = \log{\mathcal{G}(u)} + \log{\eta_m} = \hat{\mathcal{G}}(u) + \hat{\eta}_m
\]
where the logarithm is defined componentwise. We may write down the density $\hat{\rho}_m$ of $\hat{\eta}_m$:
\[
\hat{\rho}_m(y_1,\ldots,y_J) = \exp\bigg(\sum_{j=1}^J y_j\bigg)\rho_m(e^{y_1},\ldots,e^{y_J}).
\]
Then we have that the density $\hat{\rho}(\hat{y}|u)$ of the law of $\hat{y}|u$ is given by $\hat{\rho}(\hat{y}|u) = \hat{\rho}_m(\hat{y}-\hat{\mathcal{G}}(u))$. The negative log-likelihood is then given in terms of the original data as
\begin{align*}
-\log\hat{\rho}(\hat{y}|u) &= -\log\hat{\rho}_m(\log(y)-\hat{\mathcal{G}}(u))\\
&= -\sum_{j=1}^J(\log{y_j}-\log{\mathcal{G}_j(u)})  -\log\rho_m(e^{\log(y_1)-\log\mathcal{G}_1(u)},\ldots,e^{\log(y_J)-\log\mathcal{G}_J(u)})\\
&=\log\det G(u)- \log\rho_m(G(u)^{-1}y)
\end{align*}
after dropping the terms that depend only on the data. This coincides with $\Phi$ as given above.
\end{remark}

\begin{example}
\label{ex:mult}
A commonly used choice of distribution in image denoising is such that each component of $\eta_m$ is independently and identically Gamma distributed with mean $1$ and precision $\alpha > 0$ so that
\[
\rho_m(y) \propto \prod_{j=1}^J y_j^{\alpha-1}\exp(-\alpha y_j)\mathds{1}_{(0,\infty)}(y_j).
\]
In this case, we have (up to addition of purely data-dependent terms)
\[
\Phi(u;y) = \alpha\sum_{j=1}^J\left(\log\mathcal{G}_j(u) + \frac{y_j}{\mathcal{G}_j(u)}\right).
\]
Suppose that $\G(u)$ is a function on a domain $D\subseteq\R^d$, and $\{\G_j(u)\}$ is a sequence of equidistributed point evaluations of $\G(u)$, with corresponding observations $\{y_j\}$. Then with the choice $\alpha = \alpha_0/J$, so that noise level is increased as more observations are taken, we see that
\[
\Phi(u;y) \to \alpha_0\int_D \left(\log\G(u)+ \frac{y}{\G(u)}\right)
\]
as $J\to\infty$. This is precisely the form of misfit \cref{eq:gamma}.
\end{example}

When we do not have any additive noise, we require that $\mathcal{G}$ is bounded away from zero. If $\mathcal{G}(u)$ were allowed to approach zero, then the data could approach infinite precision without modifying the variance of the noise -- the posterior is then singular with respect to the prior instead of absolutely continuous, preventing well-posedness in a number of metrics.

We place the following assumptions on the noise, the prior and $\mathcal{G}$.

\begin{assumptions}
\label{assump:mult}
There exists $X'\subseteq X$ such that
\begin{enumerate}[(i)]
\item $\mathcal{G}$ is polynomially bounded and bounded away from zero on $X'$: there exist $C > 0$, $p \geq 1$ and $\eps > 0$ such that for all $u \in X'$,
\[
\eps \leq |\mathcal{G}(u)|_Y \leq C(1 + \|u\|_X^p).
\]
\item $\mathcal{G}$ is Lipschitz on bounded subsets of $X'$: for every $u_1,u_2 \in X'$ with $\|u_1\|_X,\|u_2\|_X < r$, there exists $L(r) > 0$ such that
\[
|\mathcal{G}(u_1)-\mathcal{G}(u_2)|_Y \leq L(r)\|u_1-u_2\|_X.
\]
\item $\rho_m$ is Lipschitz and bounded, with support given by
\[
\supp(\rho_m) = A_1\times\ldots\times A_J,
\]
where each $A_j \in \{\mathbb{R},(-\infty,0],[0,\infty)\}$. Additionally there exist $C,q > 0$ such that for all $r>0$ sufficiently small, 
\[
\inf\{\rho_m(z)\;|\;z \in \supp(\rho_m), |z| = r\} \geq Cr^q.
\]
\item $\mu_0(X') = 1$ and for all $k \in \mathbb{N}$,
\[
\int_{X'}\|u\|_X^k\,\mu_0(\dee u) < \infty.
\]
\end{enumerate}
\end{assumptions}

It is straightforward to verify that the Gamma and Gaussian distribution in \cref{ex:mult} satisfy \cref{assump:mult}(iii). The above assumptions mean that the potential $\Phi$ has the following properties. These properties are essentially the assumptions in section 4 of \cite{DS17} used to establish existence and well-posedness of the posterior, combined with the additional assumptions the provide existence of MAP estimators \cite{DLSV13}. In what follows, denote $Y' = \intr(\supp(\rho_m))$.

\begin{proposition}
\label{prop:wpmult}
Let \cref{assump:mult} hold. Then $\Phi:X'\times Y' \rightarrow \mathbb{R}$ satisfies:
\begin{enumerate}[(i)]
\item There exists a function $M_0:\mathbb{R}^+\times\mathbb{R}^+\rightarrow\mathbb{R}^+$ monotonic non-decreasing separately in each argument such that for all $u \in X'$ and $y \in Y'$,
\[
\Phi(u;y) \leq M_0(|y|_Y,\|u\|_X).
\]
Moreover, $\exp(M_0(|y|_Y,\|u\|_X))$ is polynomially bounded in $u$ for all $y \in Y$.
\item There exist functions $M_i:\mathbb{R}^+\times\mathbb{R}^+\rightarrow\mathbb{R}^+$, $i=1,2$ monotonic non-decreasing separately in each argument, and with $M_2$ strictly positive, such that for all $u \in X'$, $y,y_1,y_2 \in Y'\cap B_Y(0,r)$,
\begin{align*}
\Phi(u;y) &\geq -M_1(r,\|u\|_X),\\
|\Phi(u;y_1)-\Phi(u;y_2)| &\leq M_2(r,\|u\|_X)|y_1-y_2|_Y.
\end{align*}
\item For each $y \in Y'$ and $u_1,u_2 \in X'$ with $\|u_1\|_X,\|u_2\|_X < r$, there exists $M_3:\mathbb{R}^+\times\mathbb{R}^+\rightarrow\mathbb{R}^+$ such that
\[
|\Phi(u_1;y) - \Phi(u_2;y)| \leq M_3(|y|_Y,r)\|u_1-u_2\|_X.
\]
\end{enumerate}
\end{proposition}

\begin{proof}
\begin{enumerate}[(i)]
\item Let $y \in Y'$. Observe that since $G(u)$ acts on $y$ componentwise, with each $\mathcal{G}_j(u) > 0$, $G(u)^{-1}y$ lies in the same orthant as $y$: for each $u$, $G(u)^{-1}y$ is a rescaling of $y$ along a line through the origin. By the bounds on $\mathcal{G}$ we have
\begingroup
\allowdisplaybreaks
\begin{align*}
\rho(y|u) &= \frac{1}{\det G(u)}\rho_m(G(u)^{-1}y)\\
&\geq \frac{1}{C(1+\|u\|_X^{pJ})}\rho_m(G(u)^{-1}y)\\
&= \frac{1}{C(1+\|u\|_X^{pJ})}\inf\big\{\rho_m(z)\;|\;z \in Y', z/|z|_Y = G(u)^{-1}y/|G(u)^{-1}y|_Y,\\
&\hspace{9cm}|z|_Y = |G(u)^{-1}y|_Y\big\}\\
&\geq \frac{1}{C(1+\|u\|_X^{pJ})}\inf\bigg\{\rho_m(z)\,\bigg|\,z \in Y', z/|z|_Y = G(u)^{-1}y/|G(u)^{-1}y|_Y,\\
&\hspace{7.2cm}\frac{|y|_Y}{C(1+\|u\|_X^p)} \leq |z|_Y \leq |y|_Y/\eps \bigg\}\\
&\geq \frac{1}{C(1+\|u\|_X^{pJ})}\inf\bigg\{\rho_m(z)\,\bigg|\,z \in Y', z/|z|_Y = G(u)^{-1}y/|G(u)^{-1}y|_Y,\\
&\hspace{6.8cm}\frac{\min\{|y|_Y,|y|_Y^{-1}\}}{C(1+\|u\|_X^p)} \leq |z|_Y \leq |y|_Y/\eps \bigg\}\\
&=: m_0(|y|_Y,\|u\|_X).
\end{align*}
\endgroup
Since $\eps > 0$, the infimum is taken over a compact line segment contained within $Y'$. $\rho_m$ is positive and continuous everywhere, hence $m_0$ is positive. Moreover it is non-increasing in both arguments -- the min term ensures this. We then have
\[
\Phi(u;y) = -\log\rho(y|u) \leq -\log m_0(|y|_Y,\|u\|_X) =: M_0(|y|_Y,\|u\|_X).
\]
It remains to show that $\exp(M_0(|y|_Y,\|u\|_X))$ is polynomially bounded in $u$. Since the lower bound in the infimum is the only part that depends on $u$, it suffices to show that
\[
h(u) := \inf\left\{\rho_m(z)\,\bigg|\,|z|_Y = \frac{\min\{|y|_Y,|y|_Y^{-1}\}}{C(1+\|u\|_X^p)}\right\}^{-1}
\]
is polynomially bounded.  \cref{assump:mult}(iii) then tells us that for large enough $\|u\|_X$ we have
\[
h(u) \leq C\left(\frac{\min\{|y|_Y,|y|_Y^{-1}\}}{1+\|u\|_X^p}\right)^{-q}
\]
which gives the required boundedness.
\item Let $u \in X'$ and $y \in Y'$ with $|y|_Y < r$. By the boundedness of $\rho_m$, we have $\rho(y|u) \leq C\eps^{-J}$ and so
\begin{align*}
\Phi(u;y) = -\log\rho(y|u) \geq J\log\eps - \log C =: M_1(r,\|u\|_X).
\end{align*}
For the Lipschitz property, let $y_1, y_2 \in Y'$ with $y_1, y_2 < r$. Then we have by the mean value theorem
\begin{align*}
|\Phi(u;y_1) - \Phi(u;y_2)| &= |\log\rho(y_1|u) - \log\rho(y_2|u)|\\
&\leq \frac{1}{\inf_{|y|_Y<r} \rho(y|u)}|\rho(y_1|u) - \rho(y_2|u)|.
\end{align*}
Using the bound from part (i),
\[
\rho(y|u) \geq \exp(-M_0(|y|_Y,\|u\|_X)) \geq \exp(-M_0(r,\|u\|_X))
\]
and so since $\rho_m$ is Lipschitz,
\begin{align*}
|\Phi(u;y_1) - \Phi(u;y_2)| &\leq \exp(M_0(r,\|u\|_X))|\rho(y_1|u) - \rho(y_2|u)|\\
&\leq C\exp(M_0(r,\|u\|_X))\frac{1}{\det G(u)}|G(u)^{-1}(y_1-y_2)|_Y\\
&\leq C\exp(M_0(r,\|u\|_X))\eps^{-J}\|G(u)^{-1}\|_{\mathrm{op}}|y_1-y_2|_Y\\
&\leq  C\exp(M_0(r,\|u\|_X))\eps^{-J-1}|y_1-y_2|_Y\\
&=: M_2(r,\|u\|_X)|y_1-y_2|_Y.
\end{align*}

\item Let $u \in X'$ with $\|u\|_X < r$, and let $y \in Y'$. Then similarly to the above we have by the Lipschitz property of $\rho_m$ and $\mathcal{G}$,
\begin{align*}
|\Phi(u_1;y) - \Phi(u_2;y)| &\leq \exp(M_0(|y|_Y,r))|\rho(y|u_1) - \rho(y|u_2)|\\
&\leq C\exp(M_0(|y|_Y,r))\eps^{-J}|G(u_1)^{-1}y - G(u_2)^{-1}y|_Y\\
&\leq C\exp(M_0(|y|_Y,r))\eps^{-J}|y|_Y\|G(u_1)^{-1} - G(u_2)^{-1}\|_{\mathrm{op}}\\
&= C\exp(M_0(|y|_Y,r))\eps^{-J}|y|_Y\max_j |\mathcal{G}_j(u_1)^{-1} - \mathcal{G}_j(u_2)^{-1}|\\
&\leq C\exp(M_0(|y|_Y,r))\eps^{-J-2}|y|_Y\max_j |\mathcal{G}_j(u_1) - \mathcal{G}_j(u_2)|\\
&\leq C\exp(M_0(|y|_Y,r))\eps^{-J-2}|y|_Y|\mathcal{G}(u_1) - \mathcal{G}(u_2)|_Y\\
&\leq C\exp(M_0(|y|_Y,r))\eps^{-J-2}|y|_Y L(r)\|u_1 - u_2\|_X\\
&=: C(r)\|u_1 - u_2\|_X.
\end{align*}
\end{enumerate}
\end{proof}

\begin{theorem}[Existence]
\label{thm:exist_mult}
Let \cref{assump:mult} hold. Assume that $\mu_0(X') = 1$ and that $\mu_0(X\cap B) > 0$ for some bounded set $B$ in $X$. Let $Y'$ denote the interior of the support of $\rho_m$. Then for every $y \in Y'$,
\[
Z(y) := \int_X \exp(-\Phi(u;y))\,\mu_0(\dee u) > 0
\]
and the probability measure $\mu^y$ on $X$ given by
\[
\frac{\dee \mu^y}{\dee \mu_0}(u) = \frac{1}{Z(y)}\exp(-\Phi(u;y))
\]
is well-defined.
\end{theorem}

\begin{proof}
This follows from an application of Theorem 4.3 in \cite{DS17}. It is required that $\exp(M_1(r,\|u\|_X)) \in L^1_{\mu_0}(X;\mathbb{R})$ for all $r > 0$, which holds in our case since $M_1$ is constant.
\end{proof}

\begin{theorem}[Well-posedness]
Let \cref{assump:mult} hold. Assume that $\mu_0(X') = 1$ and that $\mu_0(X\cap B) > 0$ for some bounded set $B$ in $X$. Let $Y'$ denote the interior of the support of $\rho_m$. Then there is a $C = C(r) > 0$ such that, for all $y,y' \in B_Y(0,r)\cap Y'$,
\[
d_{\mathrm{Hell}}(\mu^{y},\mu^{y'}) \leq C|y-y'|_Y.
\]
\end{theorem}

\begin{proof}
This follows from an application of Theorem 4.5 in \cite{DS17}. It is required that $\exp(M_1(r,\|u\|_X))\left(1 + M_2(r,\|u\|_X)^2\right) \in L^1_{\mu_0}(X;\mathbb{R})$ for all $r > 0$. The previous proposition shows that $M_1$ is constant, and $M_2(r,\|u\|_X) = C\exp(M_0(r,\|u\|_X))$ is polynomially bounded in $u$. The result follows from the assumption that $\mu_0$ has all polynomial moments.
\end{proof}

In the case that the prior $\mu_0$ is Gaussian, we can use the result of  \cref{prop:wpmult} to prove existence of and characterize the modes of the posterior distribution. First recall the definition of MAP estimator in the general Banach space setting \cite{DLSV13}:
\begin{definition}
Let $\mu$ be a measure on a Banach space $X$, and let
\[
z^\delta = \underset{z \in X}{\mathrm{argmax}}\;\mu(B_X(z,\delta)).
\]
Any point $\tilde{z} \in X$ satisfying
\[
\lim_{\delta\downarrow 0} \frac{\mu(B_X(\tilde{z},\delta))}{\mu(B_X(z^\delta,\delta))} = 1
\]
is a MAP estimator for the measure $\mu$.
\end{definition}

The MAP estimators of the posterior measure $\mu^y$ are characterized as minimizers of a certain functional:

\begin{theorem}[Existence of MAP estimators]
\label{thm:map_mult}
Let $\mu_0$ be a Gaussian measure, and let \cref{assump:mult} hold. Let $(E,\|\cdot\|_E)$ denote the Cameron-Martin space associated with the prior measure $\mu_0$. Let $y \in Y'$ and define the functional $I:X\rightarrow\mathbb{R}$ by
\begin{align}
\label{eq:om}
I(u) = 
\begin{cases}
\Phi(u;y) + \frac{1}{2}\|u\|_E^2 & u \in E\\
\infty & u \notin E
\end{cases}
\end{align} 
Then minimizers of $I$ exist in $E$. Moreover, any minimizer of $I$ is a MAP estimator of the posterior measure $\mu^y$, and any MAP estimator of $\mu^y$ minimizes $I$.
\end{theorem}

\begin{proof}
First note that by Fernique's theorem, \cref{assump:mult}(iv) does indeed hold. The result of the theorem follows from Corollary 3.10 in \cite{DLSV13} once we show that the required assumptions hold. Namely for all $y \in Y'$, $\Phi(\cdot;y)$ is bounded below uniformly by a constant, bounded on bounded sets, and locally Lipschitz. These all follow from  \cref{prop:wpmult}.
\end{proof}

\subsection{Mixed noise model}
We now look at the case where the observations are corrupted by both additive and multiplicative noise. We assume that the data $y \in Y$ is modelled by
\begin{align}
\label{eqn:mixednoise}
y = \eta_m\mathcal{G}(u) + \eta_a,
\end{align}
where the distributions of $\eta_m$ and $\eta_a$ admit respective Lebesgue densities $\rho_m$ and $\rho_a$. Furthermore we assume that $\eta_m$ and $\eta_a$ are independent. It can be verified that the density of conditional distribution $y|u$ takes the following form.

\begin{proposition}
\label{prop:mixed_density}
Let $y$ be given by \cref{eqn:mixednoise} and assume that $u$ is independent of the noise. Then the density $\rho(y|u)$ of the conditional random distribution $y|u$ exists and is given by
\[
\rho(y|u) = \int_{Y} \rho_a(y - G(u)x)\rho_m(x)\,\dee x,
\]
where $G(u) = \mathrm{diag}(\mathcal{G}_j(u))$.
\end{proposition}

The negative log-likelihood $\Phi:X\times Y\rightarrow\mathbb{R}$ is hence given by
\[
\Phi(u;y) = -\log\left(\int_{Y} \rho_a(y - G(u)x)\rho_m(x)\,\dee x\right).
\]

If the observations are all independent, i.e., $\rho_m$ and $\rho_a$ factor into products over each coordinate, this may be written
\[
\Phi(u;y) = -\sum_{j=1}^J \log\left(\int_{\mathbb{R}} \rho_a^j(y_j - \mathcal{G}_j(u)x_j)\rho_m^j(x_j)\,\dee x_j\right).
\]
We state the assumptions we place on $\rho_m$, $\rho_a$ and $\mathcal{G}(u)$. Note that we no longer require that $\mathcal{G}$ is bounded away from zero.
\begin{assumptions}
\label{assump:mixed}
\begin{enumerate}[(i)]
\item $\mathcal{G}$ is polynomially bounded: there exist $X'\subseteq X$, $C_1 >0$, and $p \geq 1$ such that for all $u \in X'$,
\[
|\mathcal{G}(u)|_Y \leq C_1(1 + \|u\|_X^p).
\]
Furthermore $\mathcal{G}$ is Lipschitz on bounded sets.
\item $\log\rho_a$ and $\nabla\log\rho_a$ are polynomially bounded: there exist $C_2,C_3 > 0$, $q_1,q_2 \geq 1$ such that for all $y \in Y$,
\begin{align*}
|\log\rho_a(y)| &\leq C_2(1 + |y|_Y^{q_1}),\;\;\; |\nabla\log\rho_a(y)|_Y \leq C_3(1 + |y|_Y^{q_2}).
\end{align*}
Furthermore $\rho_a$ is Lipschitz and bounded.
\item $\eta_m$ has moments up to order $q = \max\{q_1,q_2\}$, i.e.
\[
\int_{Y} |x|^q_Y\rho_m(x)\,\dee x < \infty.
\]
\item Let $y \in Y$ and let $A \in \mathbb{R}^{J\times J}$ be a diagonal matrix. Define the probability measure $\nu^{y,A}$ on $Y$ by
\[
\nu^{y,A}(\dee x) \propto \rho_a(y-Ax)\rho_m(x)\,\dee x.
\]
Then, with $q_2$ as above, there exist $r_1, r_2 \geq 0$ such that
\[
\int_Y |x|^{q_2}_Y\,\nu^{y,A}(\dee x) \leq C(1 + |y|^{r_1}_Y + \|A\|_{\mathrm{op}}^{r_2}).
\]
\item $\mu_0$ has all polynomial moments: for all $k \in \mathbb{N}$,
\[
\int_{X'}\|u\|_X^k\,\mu_0(\dee u) < \infty.
\]
\end{enumerate}
\end{assumptions}

\begin{remark}
\cref{assump:mixed}(iv) above is used in the proof of well-posedness of the posterior distribution. It is a complicated condition to check, and may be replaced by the stronger condition that $\rho_m$ is compactly supported. Then we have that
\begin{align*}
\int_Y |x|^{q_2}_Y\,\nu^{y,A}(\dee x) &= \frac{\int_Y |x|^{q_2}_Y \rho_a(y-Ax)\rho_m(x)\,\dee x}{\int_Y \rho_a(y-Ax)\rho_m(x)\,\dee x} \leq \max_{x \in \supp(\rho_m)}|x|^{q_2}_Y
\end{align*}
providing a uniform bound. This is a strong and far from necessary condition however, and rules out the useful case where $\eta_m$ is Gaussian.
\end{remark}

We give the analogous result to \cref{prop:wpmult} for the mixed noise case.

\begin{proposition}
\label{prop:wpmixed}
Let \cref{assump:mixed} hold. Then $\Phi:X'\times Y \rightarrow \mathbb{R}$ satisfies
\begin{enumerate}[(i)]
\item There exists a function $M_0:\mathbb{R}^+\times\mathbb{R}^+\rightarrow\mathbb{R}^+$ monotonic non-decreasing separately in each argument such that for all $u \in X'$ and $y \in Y$,
\[
\Phi(u;y) \leq M_0(|y|_Y,\|u\|_X).
\]
\item There exist functions $M_i:\mathbb{R}^+\times\mathbb{R}^+\rightarrow\mathbb{R}^+$, $i=1,2$ monotonic non-decreasing separately in each argument, and with $M_2$ strictly positive, such that for all $u \in X'$, $y,y_1,y_2 \in B_Y(0,r)$,
\begin{align*}
\Phi(u;y) &\geq -M_1(r,\|u\|_X),\\
|\Phi(u;y_1)-\Phi(u;y_2)| &\leq M_2(r,\|u\|_X)|y_1-y_2|_Y.
\end{align*}
\item For each $y \in Y$ and $u_1,u_2 \in X'$ with $\|u_1\|_X,\|u_2\|_X < r$, there exists $M_3:\mathbb{R}^+\times\mathbb{R}^+\rightarrow\mathbb{R}^+$ such that
\[
|\Phi(u_1;y) - \Phi(u_2;y)| \leq M_3(|y|_Y,r)\|u_1-u_2\|_X.
\]
\end{enumerate}
\end{proposition}

\begin{proof}
\begin{enumerate}[(i)]
\item Let $u \in X'$ and $y \in Y$. Using Jensen's inequality with the convex function $-\log$ and the probability measure $\rho_m(x)\,\dee x$, we have
\begin{align*}
\Phi(u;y) &= -\log\left(\int_Y \rho_a(y - G(u)x)\rho_m(x)\,\dee x\right)\\
&\leq \int_Y -\log(\rho_a(y - G(u)x))\rho_m(x)\,\dee x.
\end{align*}
Hence using the assumptions on polynomial growth and existence of moments,
\begin{align*}
\Phi(u;y) &\leq \int_{Y} C(1 + |y - G(u)x|_Y^{q_1})\rho_m(x)\,\dee x\\
&\leq \int_{Y} C(1 + |y|_Y^{q_1} + (1+\|u\|_X^{pq_1})|x|^{q_1}_Y)\rho_m(x)\,\dee x\\
&\leq C(1 + |y|_Y^{q_1} + \|u\|_X^{pq_1})\\
&=: M_0(|y|_Y, \|u\|_X).
\end{align*}

\item Let $u \in X'$ and let $y \in Y$ with $|y|_Y < r$. By the boundedness of $\rho_a$, we have
\[
\rho(y|u) = \int_Y \rho_a(y - G(u)x)\rho_m(x)\,\dee x \leq C\int_Y\rho_m(x)\,\dee x = C
\]
and so
\[
\Phi(u;y) = -\log\rho(y|u) \geq -\log C =: M_1(r,\|u\|_X).
\]
For the Lipschitz property, let $y_1, y_2 \in Y$ with $|y_1|_Y,|y_2|_Y < r$, and let $u \in X$. We may differentiate $\Phi(u;\cdot)$ to see that
\begin{align*}
|\nabla_y \Phi(u;y)|_Y &= |\nabla_y \log\rho(y|u)|_Y\\
&\leq \frac{1}{\rho(y|u)}\int_Y |\nabla\rho_a(y-G(u)x)|_Y\rho_m(x)\,\dee x\\
&= \frac{1}{\rho(y|u)}\int_Y |\nabla\log\rho_a(y-G(u)x)|_Y\rho_a(y-G(u)x)\rho_m(x)\,\dee x\\
&\leq \frac{1}{\rho(y|u)}\int_Y C(1+|y|_Y^{q_2} + (1+\|u\|_X^{pq_2})|x|^{q_2}_Y)\rho_a(y-G(u)x)\rho_m(x)\,\dee x\\
&\leq C(1+|y|_Y^{q_2}) + C(1+\|u\|_X^{pq_2})\frac{\int_Y |x|^{q_2}_Y\rho_a(y-G(u)x)\rho_m(x)\,\dee x}{\int_Y \rho_a(y-G(u)x)\rho_m(x)\,\dee x}\\
&\leq C(1+|y|_Y^{s_1} + \|u\|_X^{s_2})
\end{align*}
for some $s_1,s_2 \geq 0$, where we have used the polynomial bounds on $\nabla\log\rho_a$ and $\mathcal{G}$, as well as \cref{assump:mixed}(iv). Hence for $u \in X'$ and $y_1,y_2 \in Y$ with $|y_1|_Y,|y_2|_Y < r$,
\begin{align*}
|\Phi(u;y_1)-\Phi(u;y_2)| &\leq \sup_{|y|_Y<r}|\nabla_y\Phi(u;y)|_Y|y_1-y_2|_Y\\
&\leq C(1+r^{s_1} + \|u\|_X^{s_2})|y_1-y_2|_Y\\
&=: M_2(r,\|u\|_X)|y_1-y_2|_Y.
\end{align*}

\item Let $u_1, u_2 \in X$ with $\|u_1\|_X, \|u_2\|_X < r$, and let $y \in Y$. Then similarly to the purely multiplicative case, using the Lipschitz property of $\rho_a$ and $\mathcal{G}$,
\begin{align*}
|\Phi(u_1;y) - \Phi(u_2;y)| &\leq \exp(M_0(|y|_Y,r))\int_{Y}C|(G(u_1) - G(u_2))x|_Y\rho_m(x)\,\dee x\\
&\leq C \exp(M_0(|y|_Y,r))\int_{Y}L(r)\|u_1-u_2\|_X|x|_Y\rho_m(x)\,\dee x\\
&=: M_3(|y|_Y,r)\|u_1-u_2\|_X.
\end{align*}
\end{enumerate}
\end{proof}

The case when both the multiplicative and additive noise are Gaussian is covered by the above set-up: 
\begin{proposition}
\label{prop:gauss_assump}
Let $\Gamma^a, \Gamma^m \in \mathbb{R}^{J\times J}$ be positive definite matrices, and suppose that $\eta_a \sim N(0,\Gamma^a)$ and $\eta_m \sim N(\mathbf{1},\Gamma^m)$ are both Gaussian. Then \cref{assump:mixed}(ii)-(iv) are satisfied.
\end{proposition}

The proof of this proposition is given in the appendix, however the lengthy calculation in the proof to show that \cref{assump:mixed}(iv) holds can actually be avoided by noting that the likelihood $\rho(y|u)$ is again Gaussian, from which the polynomial bound on the Lipschitz constant $M_2$ can be seen directly. Moreover, this allows us to relax the assumption on $\Gamma^m$ from positive definite to positive semi-definite.

\begin{theorem}[Existence]
Let \cref{assump:mixed} hold. Assume that $\mu_0(X') = 1$ and that $\mu_0(X\cap B) > 0$ for some bounded set $B$ in $X$. Then for every $y \in Y$, the conclusions of \cref{thm:exist_mult} hold.
\end{theorem}

\begin{proof}
The proof is identical to the multiplicative case, since again we have that $M_1$ is constant.
\end{proof}

\begin{theorem}[Well-posedness]
Let \cref{assump:mixed} hold. Assume that $\mu_0(X') = 1$ and that $\mu_0(X\cap B) > 0$ for some bounded set $B$ in $X$. Then there is a $C = C(r) > 0$ such that, for all $y,y' \in B_Y(0,r)$,
\[
d_{\mathrm{Hell}}(\mu^{y},\mu^{y'}) \leq C|y-y'|_Y.
\]
\end{theorem}

\begin{proof}
This is again identical to the multiplicative case since the proof of \cref{prop:wpmixed} shows that $M_1$ is constant and $M_2$ is polynomially bounded in $u$.
\end{proof}

As before, in the case that the prior $\mu_0$ is Gaussian, we can use the result of \cref{prop:wpmixed} to prove existence of and characterize the modes of the posterior distribution. 

\begin{theorem}[Existence of MAP estimators]
\label{thm:map_mixed}
Let $\mu_0$ be a Gaussian measure, and let \cref{assump:mixed} hold. Let $(E,\|\cdot\|_E)$ denote the Cameron-Martin space associated with the prior measure $\mu_0$. Let $y \in Y$ and define the functional $I:X\rightarrow\mathbb{R}$ by \cref{eq:om}. Then minimizers of $I$ exist in $E$. Moreover, any minimizer of $I$ is a MAP estimator of the posterior measure $\mu^y$, and any MAP estimator of $\mu^y$ minimizes $I$.
\end{theorem}

\begin{proof}
The proof is the same as the multiplicative case, with \cref{prop:wpmixed} in place of \cref{prop:wpmult}.
\end{proof}

\begin{example}[Gaussian noise]
\label{ex:gaussian}
Let $\Gamma^a, \Gamma^m \in \mathbb{R}^{J\times J}$ be positive definite matrices, and suppose that $\eta_a \sim N(0,\Gamma^a)$ and $\eta_m \sim N(\mathbf{1},\Gamma^m)$ are both Gaussian. It was shown earlier that \cref{assump:mixed}(ii)-(iv) are satisfied by these densities.

Assuming $u$ is independent of the noise, it can be checked that $y|u$ is also Gaussian, with mean $\mathcal{G}(u)$ and covariance matrix
\[
\Gamma(u) := \Gamma^a + \mathcal{G}(u)\mathcal{G}(u)^*\circ\Gamma^m,
\]
where $\circ$ denotes the Hadamard (componentwise) product of matrices, i.e.
\[
\Gamma(u)_{ij} = \Gamma^a_{ij} + \mathcal{G}_i(u)\mathcal{G}_j(u)\Gamma^m_{ij},\;\;\;i,j=1,\ldots,J.
\]
We can then see directly, without performing any integration, that the negative log-likelihood $\Phi$ is given by
\[
\Phi(u;y) = \frac{1}{2}|\mathcal{G}(u)-y|_{\Gamma(u)}^2 + \frac{1}{2}\log\det\Gamma(u).
\]
If we take the observations to be independent, the covariance matrices are diagonal and can be written in the form
\[
\Gamma^a = \mathrm{diag}(\gamma_{a,j}^2),\;\;\;\Gamma^m = \mathrm{diag}(\gamma_{m,j}^2)
\]
so that
\[
\Gamma(u) = \mathrm{diag}(\gamma_{a,j}^2 + \gamma_{m,j}^2\mathcal{G}_j(u)^2).
\]
Then the potential $\Phi$ is given by
\[
\Phi(u;y) = \frac{1}{2}\sum_{j=1}^J \frac{|\mathcal{G}_j(u) - y_j|^2}{\gamma_{a,j}^2 + \gamma_{m,j}^2\mathcal{G}_j(u)^2} + \frac{1}{2}\sum_{j=1}^J\log(\gamma_{a,j}^2 + \gamma_{m,j}^2\mathcal{G}_j(u)^2).
\]
Suppose now that, as in \cref{ex:mult}, $\G(u)$ is a function on domain $D\subseteq\R^d$, with $\{\G_j(u)\}$ being a sequence of equidistributed point evaluations of $\G(u)$ and $\{y_j\}$ the corresponding observations. Scaling the variances $\gamma_{a,j}^2 \equiv J\gamma_{a}^2$, $\gamma_{m,j}^2 \equiv J\gamma_{m}^2$, we see that (up to constants)
\[
\Phi(u;y) \to \frac{1}{2}\int_D \left(\frac{|\G(u)-y|^2}{\gamma_a^2 + \gamma_m^2\G(u)^2} + \log(\gamma_a^2 + \gamma_m^2\G(u)^2)\right)
\]
as $J\to\infty$. The denominator here could be viewed as a type of self-normalization of the data. This has a similar form to the misfit that appears in some variational approaches to inverse problems wherein the $L^2$ misfit is normalized by the size of the observed data --  the $\G(u)$ in the denominator is replaced by $y$ \cite{isaac2015scalable}. The logarithmic term is typically not present when this is the case. This could be viewed as approximating the mixed noise model by an additive noise model, allowing for easier optimization; however the corresponding data model would be ill-posed in the sense that the variance of the additive noise depends on its own realization.
\end{example}

\section{Consistency of MAP estimators: mixed Gaussian noise}
The solution to the inverse problem is given in terms of a probability distribution. If the data arises from a true state $u^\dagger \in X$, we would like the posterior to concentrate on this state in some sense, in the limit of small noise or large data. As in \cite{DLSV13}, we can show weak convergence of MAP estimators to states $u^*$ with $\mathcal{G}(u^*) = \mathcal{G}(u^\dagger)$, assuming that $u^\dagger \in E$. If $u^\dagger$ is only in $X$, we get the corresponding weaker result.

We concentrate on the case where the observations are corrupted by both multiplicative and additive Gaussian noise, introduced in \cref{ex:gaussian}. We take a Gaussian prior $\mu_0$, with Cameron-Martin space $(E,\langle\cdot,\cdot\rangle_E)$. We provide full details for the small noise limit, and some motivating arguments for the large data limit. As in the previous section, throughout this section $C$ denotes a generic constant.

\subsection{Small noise limit}

First assume that the true state $u^\dagger \in E$ . Suppose we have a sequence of independent observations $(y_n)_{n\in\mathbb{N}}$ in which the variances of the additive and multiplicative noise diminish at the same rate:
\begin{align*}
y_n = \left(\mathbf{1}+\frac{1}{n}\eta^m_n\right)\mathcal{G}(u^\dagger) + \frac{1}{n}\eta^a_n
\end{align*}
where $\eta^m_n \sim N(0,\Gamma^m)$ and $\eta^a_n \sim N(0,\Gamma^a)$ are independent realisations of the noise. As in \cref{ex:gaussian}, we assume $\Gamma^a$ and $\Gamma^m$ are positive definite.

The conditional distribution of $y_n|u$ is $N(\mathcal{G}(u),\Gamma_n(u))$, where $\Gamma_n(u) = \Gamma(u)/n^2$ and $\Gamma(u) = \Gamma^a + \mathcal{G}(u)\mathcal{G}(u)^*\circ\Gamma^m$ is as defined in \cref{ex:gaussian}. The posterior distribution is then given by
\[
\frac{\dee \mu^{y_n}}{\dee \mu_0}(u) \propto \exp\left(-\frac{n^2}{2}|\mathcal{G}(u)-y_n|^2_{\Gamma(u)} - \frac{1}{2}\log\det\Gamma(u)\right)
\]
and so the negative log-likelihood is
\[
\Phi_n(u;y) = \frac{n^2}{2}|\mathcal{G}(u)-y_n|^2_{\Gamma(u)} + \frac{1}{2}\log\det\Gamma(u).
\]
It can be seen that the conclusion of \cref{thm:map_mixed} holds for each $n$, and so we know that MAP estimators of the posterior $\mu^{y_n}$ exists and are minimizers of the functional
\[
I_n(u) =  \frac{1}{2}\|u\|^2_E + \frac{n^2}{2}|\mathcal{G}(u)-y_n|^2_{\Gamma(u)} + \frac{1}{2}\log\det\Gamma(u).
\]
We wish to study the behaviour of minimizers of these functionals as $n$ gets large. To do so it is instructive (though not necessary) to first look at the limit of the functionals itself. In order to obtain a non-trivial limit we rescale $I_n$ by $2/n^2$, noting that this does not affect the minimizers, i.e.
\[
\underset{u \in E}{\mathrm{argmin}}\,I_n(u) = \underset{u \in E}{\mathrm{argmin}}\,\frac{2}{n^2} I_n(u)\;\;\;\text{ for all }n \in \mathbb{N}.
\]

\begin{proposition}
Define the functionals $J_n:E\rightarrow\mathbb{R}$ by
\[
J_n(u) := \frac{2}{n}I_n(u) = \frac{1}{n^2}\|u\|^2_E + |\mathcal{G}(u)-y_n|^2_{\Gamma(u)} + \frac{1}{n^2}\log\det\Gamma(u).
\]
Then for all $u \in E$, we have $\mathbb{Q}_0$-almost surely,
\begin{align*}
\mathcal{J}(u) & := \lim_{n\rightarrow\infty} J_n(u) = |\mathcal{G}(u) - \mathcal{G}(u^\dagger)|^2_{\Gamma(u)}.
\end{align*}
\end{proposition}

\begin{proof}
Since $u \in E$ is fixed, the first and final terms go to zero. We insert the expression for the data $y_n$ into the middle term and expand:
\begin{align}
\notag|\mathcal{G}(u)-y_n|^2_{\Gamma(u)} &= \left|\mathcal{G}(u) - \mathcal{G}(u^\dagger) - \frac{1}{n}\eta^m_n\mathcal{G}(u^\dagger) -\frac{1}{n}\eta^a_n\right|_{\Gamma(u)}^2\\
\label{eq:misfit1}&= |\mathcal{G}(u) - \mathcal{G}(u^\dagger)|_{\Gamma(u)}^2 + |X_n|_Y^2 - 2\langle \Gamma(u)^{-1/2}(\mathcal{G}(u) - \mathcal{G}(u^\dagger)),X_n\rangle_Y
\end{align}
where
\begin{align*}
X_n &= \Gamma(u)^{-1/2}\left(\frac{1}{n}\eta^m_n\mathcal{G}(u^\dagger) + \frac{1}{n}\eta^a_n\right) \sim N\left(0,\frac{1}{n^2}\Gamma(u)^{-1/2}\Gamma(u^\dagger)\Gamma(u)^{-1/2}\right)
\end{align*}
is a centred Gaussian random variable. Chebyshev's inequality and an application of the first Borel-Cantelli lemma then gives that the final two terms in \cref{eq:misfit1} tend to zero $\mathbb{Q}_0$-almost surely.
\end{proof}

\begin{remark}
It is clear that this limiting functional is minimized by $u \in \mathbb{E}$ if and only if $\mathcal{G}(u) = \mathcal{G}(u^\dagger)$.
\end{remark}

We now show that, along a subsequence, minimizers of $J_n$ converge to those of $\mathcal{J}$ almost surely. Before we proceed we will need the following lemma regarding eigenvalues of $\Gamma(u)$.

\begin{lemma}
\label{lem:eigbound}
Given $u \in X$, let $\lambda_{\min}(u)$ and $\lambda_{\max}(u)$ denote the smallest and largest eigenvalues of $\Gamma(u)$ respectively. Then there exist $\lambda_{\min}^* > 0$ and $C > 0$ such that for all $u \in X$,
\[
\lambda_{\min}^* \leq \lambda_{\min}(u) \leq \lambda_{\max}(u) \leq C(1+|\mathcal{G}(u)|_Y^2).
\]
\end{lemma}
\begin{proof}
We first show the lower bound. $\Gamma^m$ is assumed to be positive definite, and $\mathcal{G}(u)\mathcal{G}(u)^*$ is positive semi-definite. By Schur's product theorem, it follows that the Hadamard product $\mathcal{G}(u)\mathcal{G}(u)^*\circ\Gamma^m$ is positive semi-definite. Now $\Gamma^a$ is positive definite, and so there exists $\alpha > 0$ such that for any $x \in Y$,
\[
x^*\Gamma(u)x = x^*\Gamma^ax + x^*(\mathcal{G}(u)\mathcal{G}(u)^*\circ\Gamma^m)x \geq \alpha|x|_Y^2.
\]
It follows that for any $u \in X$, the smallest eigenvalue of $\Gamma(u)$ is at least as large as $\alpha$, and we denote it $\lambda_{\min}^* := \inf_{u\in X}\lambda_{\min}(u) \geq \alpha > 0$. For the upper bound, first consider the matrix $B(u) = \mathcal{G}(u)\mathcal{G}(u)^*$. Then it can be checked that $\mathcal{G}(u)$ is an eigenvector of $B(u)$ with eigenvalue $|\mathcal{G}(u)|^2$. Now $B(u)$ has rank $1$ and so this is its only non-zero eigenvalue. Furthermore it is symmetric, and so we in fact have $\|B(u)\|_{\text{op}} = |\mathcal{G}(u)|_Y^2$. $\Gamma(u)$ is also symmetric, and so by equivalence of norms on $\mathbb{R}^{J\times J}$ we have
\begin{align*}
\lambda_{\max}(u)^2 &= \|\Gamma(u)\|_{\text{op}}^2 \leq C(\|\Gamma^a\|_{\text{op}}^2 + \|B(u)\circ\Gamma^m\|_{\text{op}}^2) \leq C(1 + \|B(u)\circ\Gamma^m\|_{\text{Frob}}^2)\\
&= C\left(1 + \sum_{i,j=1}^J (\mathcal{G}_i(u)\mathcal{G}_j(u)\Gamma^m_{ij})^2\right) \leq C\left(1 + \max_{i,j}(\Gamma^m_{ij})^2\sum_{i,j=1}^J (\mathcal{G}_i(u)\mathcal{G}_j(u))^2\right)\\
&\leq C\left(1 + \|B(u)\|^2_{\text{Frob}}\right) \leq C\left(1 + \|B(u)\|^2_{\text{op}}\right) = C\left(1 + |\mathcal{G}(u)|_Y^4\right)
\end{align*}
and so $\lambda_{\max}(u) \leq C(1+|\mathcal{G}(u)|^2)$.
\end{proof}

We can now state and prove the main result concerning convergence of MAP estimators; the general structure of the proof is similar to that of Theorem 4.1 in \cite{DLSV13}.
\begin{theorem}
\label{thm:smallnoise}
For every $n \in \mathbb{N}$, let $u_n \in E$ be a minimizer of $J_n$ given above. Then there exist a $u^* \in E$ and a subsequence of $(u_n)_{n\in\mathbb{N}}$ that converges weakly to $u^*$ in $E$ almost surely. For any such $u^*$, we have $\mathcal{G}(u^*) = \mathcal{G}(u^\dagger)$. 
\end{theorem}

\begin{proof}
We expand $J_n(u)$:
\begin{align*}
J_n(u) &:= |\mathcal{G}(u^\dagger)-\mathcal{G}(u)|_{\Gamma(u)}^2 + \frac{1}{n^2}\|u\|^2_E + \frac{1}{n^2}\log\det\Gamma(u)\\
&\hspace{1cm} + \frac{2}{n}\langle \mathcal{G}(u^\dagger) - \mathcal{G}(u),\Gamma(u)^{-1}(\eta^m_n\mathcal{G}(u^\dagger) + \eta^a_n)\rangle_Y + \frac{1}{n^2} |\eta^m_n\mathcal{G}(u^\dagger)+\eta_n^a|_{\Gamma(u)}^2.
\end{align*}

Since $u_n$ minimizes $J_n$, we have that $J_n(u_n) \leq J_n(u^\dagger)$, that is
\begin{align}
J_n(u_n) \leq \frac{1}{n^2}\|u^\dagger\|_E^2 + \frac{1}{n^2}\log\det\Gamma(u^\dagger) + \frac{1}{n^2} |\eta^m_n\mathcal{G}(u^\dagger)+\eta_n^a|_{\Gamma(u^\dagger)}^2\notag,
\end{align}
from which we see
\begin{align}
|\mathcal{G}(u^\dagger)&-\mathcal{G}(u_n)|_{\Gamma(u_n)}^2 + \frac{1}{n^2}\|u_n\|^2_E + \frac{1}{n^2}\log\det\Gamma(u_n)\notag \\
&\leq \frac{1}{n^2}\|u^\dagger\|_E^2 + \frac{1}{n^2}\log\det\Gamma(u^\dagger)  + \frac{1}{n^2} |\eta^m_n\mathcal{G}(u^\dagger)+\eta_n^a|_{\Gamma(u^\dagger)}^2\label{eq:maporig}\\
&\hspace{2cm} + \frac{2}{n}\left|\langle \mathcal{G}(u^\dagger) - \mathcal{G}(u_n),\Gamma(u_n)^{-1}(\eta^m_n\mathcal{G}(u^\dagger) + \eta^a_n)\rangle_Y\right|.\notag
\end{align}

We deal first with the final term on the right hand side. By Cauchy-Schwarz and Young's inequalities, and the lower eigenvalue bound from \cref{lem:eigbound},
\begin{align*}
\frac{2}{n}|\langle \mathcal{G}&(u^\dagger) - \mathcal{G}(u_n),\Gamma(u_n)^{-1}(\eta^m_n\mathcal{G}(u^\dagger) + \eta^a_n)\rangle_Y|\\
&= \left|\left\langle\Gamma(u_n)^{-1/2}(\mathcal{G}(u^\dagger) - \mathcal{G}(u_n)),\frac{2}{n}\Gamma(u_n)^{-1/2}(\eta^m_n\mathcal{G}(u^\dagger) + \eta^a_n) \right\rangle_Y\right|\\
&\leq |\mathcal{G}(u^\dagger) - \mathcal{G}(u_n)|_{\Gamma(u_n)}\left|\frac{2}{n}\Gamma(u_n)^{-1/2}(\eta^m_n\mathcal{G}(u^\dagger) + \eta^a_n)\right|_Y\\
&\leq \frac{1}{2}|\mathcal{G}(u^\dagger) - \mathcal{G}(u_n)|_{\Gamma(u_n)}^2 + \frac{2}{n^2}\left|\Gamma(u_n)^{-1/2}(\eta^m_n\mathcal{G}(u^\dagger) + \eta^a_n)\right|_Y^2\\
&\leq \frac{1}{2}|\mathcal{G}(u^\dagger) - \mathcal{G}(u_n)|_{\Gamma(u_n)}^2 + \frac{2}{\lambda_{\min}^*n^2}\left|\eta^m_n\mathcal{G}(u^\dagger) + \eta^a_n\right|_Y^2
\end{align*}
and so taking expectations,
\begin{align*}
\frac{2}{n}\mathbb{E}\langle \mathcal{G}&(u^\dagger) - \mathcal{G}(u_n),\Gamma(u_n)^{-1}(\eta^m_n\mathcal{G}(u^\dagger) + \eta^a_n)\rangle_Y\\
&\leq \frac{1}{2}\mathbb{E}|\mathcal{G}(u^\dagger) - \mathcal{G}(u_n)|_{\Gamma(u_n)}^2 + \frac{2}{\lambda_{\min}^*n^2}\mathbb{E}\left|\eta^m_n\mathcal{G}(u^\dagger) + \eta^a_n\right|_Y^2.
\end{align*}

Rearranging the inequality \cref{eq:maporig} and taking expectations yields
\begin{align*}
\frac{1}{2}\mathbb{E}|\mathcal{G}(u^\dagger)&-\mathcal{G}(u_n)|_{\Gamma(u_n)}^2 + \frac{1}{n^2}\mathbb{E}\|u_n\|^2_E\\
&\leq \frac{1}{n^2}\|u^\dagger\|_E^2 + \frac{2}{\lambda_{\min}^*n^2}\mathbb{E}|\eta^m_n\mathcal{G}(u^\dagger)+\eta^a_n|_Y^2 + \frac{1}{n^2} \mathbb{E}|\eta^m_n\mathcal{G}(u^\dagger)+\eta_n^a|_{\Gamma(u^\dagger)}^2 \\
&\hspace{1cm} + \frac{1}{n^2}\log\det\Gamma(u^\dagger) - \frac{1}{n^2}\mathbb{E}\log\det\Gamma(u_n)\\
&\leq \frac{1}{n^2}\|u^\dagger\|_E^2 + \frac{3}{\lambda_{\min}^*n^2}\mathbb{E}|\eta^m_n\mathcal{G}(u^\dagger)+\eta^a_n|_Y^2\\
&\hspace{1cm} + \frac{1}{n^2}\log\det\Gamma(u^\dagger) - \frac{1}{n^2}\mathbb{E}\log\det\Gamma(u_n).
\end{align*}
Now the term $\mathbb{E}|\eta^m_n\mathcal{G}(u^\dagger)+\eta^a_n|_Y^2$ is the expectation of the square of the norm of a Gaussian random variable, which is finite. Since $(\eta^m_n)_{n\geq 1}$ and $(\eta^a_n)_{n\geq 1}$ are iid, this expectation is independent of $n$. To deal with the final term on the right-hand side, we use the lower bound from  \cref{lem:eigbound} to see that  $\det\Gamma(u) \geq (\lambda_{\min}^*)^J$, and so $\log\det\Gamma(u_n) \geq J\log\lambda_{\min}^*$ for all $n \in \mathbb{N}$.

We may therefore write the above as
\begin{align*}
\frac{1}{2}\mathbb{E}|\mathcal{G}(u^\dagger)&-\mathcal{G}(u_n)|_{\Gamma(u_n)}^2 + \frac{1}{n^2}\mathbb{E}\|u_n\|^2_E \leq \frac{1}{n^2}\|u^\dagger\|_E^2 + \frac{K}{n^2}
\end{align*}
for some $K > 0$. We deduce that
\begin{align*}
\mathbb{E}|\mathcal{G}(u^\dagger)-\mathcal{G}(u_n)|_{\Gamma(u_n)}^2 \rightarrow 0,\;\;\;\mathbb{E}\|u_n\|_E^2 \leq \|u^\dagger\|_E^2 + K.
\end{align*}
From this it can be shown that there exists $u^* \in E$ and a subsequence $(u_{n_k})_{k\geq 1}$ of $(u_n)_{n\geq 1}$ such that as $k\rightarrow\infty$,
\begin{align}
|\mathcal{G}(u^\dagger)-\mathcal{G}(u_{n_k})|_{\Gamma(u_{n_k})}^2\rightarrow 0&\;\;\;\text{almost surely,}\label{eq:normzero}\\
\mathbb{E}\langle u_{n_k},v\rangle_E\rightarrow\mathbb{E}\langle u^*,v\rangle_E&\;\;\;\text{for all $v \in E$}\label{eq:innerprod}.
\end{align}
See the proof of Theorem 4.1 in \cite{DLSV13} for details.

From \cref{eq:normzero} we have that $\Gamma(u_{n_k})^{-1/2}(\mathcal{G}(u^\dagger) - \mathcal{G}(u_{n_k})) \rightarrow 0$ almost surely. We show that this implies that $\mathcal{G}(u_{n_k})\rightarrow\mathcal{G}(u^\dagger)$ almost surely, i.e. that $\Gamma(u_{n_k})$ does not become degenerate in the limit. From \cref{lem:eigbound}, it follows that
\[
|\Gamma(u_{n_k})^{-1/2}(\mathcal{G}(u^\dagger) - \mathcal{G}(u_{n_k}))|_Y \geq \frac{1}{\sqrt{\lambda_{\max}(u_{n_k})}}|\mathcal{G}(u^\dagger) - \mathcal{G}(u_{n_k})|_Y \geq \frac{|\mathcal{G}(u^\dagger) - \mathcal{G}(u_{n_k})|_Y}{C(1+|\mathcal{G}(u_{n_k})|)}
\]
and so the right hand term must tend to zero almost surely. First suppose that $\mathcal{G}(u_{n_k})$ is not bounded, then there exists a subsequence along which $|\mathcal{G}(u_{n_k})|_Y\rightarrow\infty$ almost surely. It follows that
\begin{align*}
\frac{|\mathcal{G}(u^\dagger) - \mathcal{G}(u_{n_k})|_Y}{C(1+|\mathcal{G}(u_{n_k})|_Y)} &\geq \frac{||\mathcal{G}(u^\dagger)|_Y - |\mathcal{G}(u_{n_k})|_Y|}{C(1+|\mathcal{G}(u_{n_k})|_Y)} = \frac{\left|\frac{|\mathcal{G}(u^\dagger)|_Y}{|\mathcal{G}(u_{n_k})|_Y}-1\right|}{C\left(\frac{1}{|\mathcal{G}(u_{n_k})|_Y}+1\right)} \rightarrow \frac{1}{C} > 0
\end{align*}
This is a contradiction, since the left hand side tends to zero almost surely. We can hence assume that $\mathcal{G}(u_{n_k})$ remains bounded, say by $L$. Then
\[
\frac{1}{C(1+L)}|\mathcal{G}(u^\dagger) - \mathcal{G}(u_{n_k})|_Y \leq \frac{|\mathcal{G}(u^\dagger) - \mathcal{G}(u_{n_k})|_Y}{C(1+|\mathcal{G}(u_{n_k})|_Y)}\rightarrow 0\;\;\;\text{a.s.}
\]
and so $\mathcal{G}(u_{n_k})\rightarrow\mathcal{G}(u^\dagger)$ almost surely.

We conclude the proof in the same way as in \cite{DLSV13}. The convergence $\cref{eq:innerprod}$ implies that there exists a further subsequence along which $u_{n_k}\rightarrow u^*$ weakly in $E$ almost surely. Since $E$ is compactly embedded in $X$ it follows that $u_{n_k}\rightarrow u^*$ in $X$ almost surely. Then by the continuity of $\mathcal{G}$ on $X$, $\mathcal{G}(u_{n_k})\rightarrow \mathcal{G}(u^*)$ almost surely. The result follows.
\end{proof}

Suppose now that $u^\dagger \in X$. As in \cite{DLSV13}, we get the following weaker result.

\begin{corollary}
Suppose that $\mathcal{G}$ and $u_n$ satisfy the assumptions of \cref{thm:smallnoise}, and that $u^\dagger \in X$. Then there exists a subsequence of $(\mathcal{G}(u_n))_{n\in\mathbb{N}}$ converging to $\mathcal{G}(u^\dagger)$ almost surely.
\end{corollary}

\begin{proof}
Much of the proof is analogous to that of the previous theorem, and so we omit some details. Let $\eps > 0$, then since $E$ is dense in $X$ there exists $v \in E$ such that $\|u^\dagger - v\|_X \leq \eps$. Then by definition of $u_n$ we have $J_n(u_n) \leq J_n(v)$ for all $n \in \mathbb{N}$, that is
\begin{align*}
 &|\mathcal{G}(u^\dagger)-\mathcal{G}(u_n)|_{\Gamma(u_n)}^2 + \frac{1}{n^2}\|u_n\|^2_E + \frac{1}{n^2}\log\det\Gamma(u_n)\\
&\hspace{1cm} + \frac{2}{n}\langle \mathcal{G}(u^\dagger) - \mathcal{G}(u_n),\Gamma(u_n)^{-1}(\eta^m_n\mathcal{G}(u^\dagger) + \eta^a_n)\rangle_Y + \frac{1}{n^2} |\eta^m_n\mathcal{G}(u^\dagger)+\eta_n^a|_{\Gamma(u_n)}^2\\
&\leq  |\mathcal{G}(u^\dagger)-\mathcal{G}(v)|_{\Gamma(u)}^2 + \frac{1}{n^2}\|u\|^2_E + \frac{1}{n^2}\log\det\Gamma(v)\\
&\hspace{1cm} + \frac{2}{n}\langle \mathcal{G}(u^\dagger) - \mathcal{G}(v),\Gamma(v)^{-1}(\eta^m_n\mathcal{G}(u^\dagger) + \eta^a_n)\rangle_Y + \frac{1}{n^2} |\eta^m_n\mathcal{G}(u^\dagger)+\eta_n^a|_{\Gamma(v)}^2.
\end{align*}
Dropping the (positive) second and final terms on the left-hand side, rearranging, and using the same bounds as in the previous proof, we have
\begin{align*}
|\mathcal{G}(u^\dagger)-\mathcal{G}(u_n)|_{\Gamma(u_n)}^2 &\leq |\mathcal{G}(u^\dagger)-\mathcal{G}(v)|_{\Gamma(v)}^2 + \frac{1}{n^2}\|v\|_E^2 + \frac{1}{n^2}\\
&\hspace{1cm}+\frac{1}{n^2}\log\det\Gamma(v) - \frac{1}{n^2}\log\det\Gamma(u_n)\\
&\hspace{1cm}+\frac{1}{2}|\mathcal{G}(u^\dagger) - \mathcal{G}(u_n)|_{\Gamma(u_n)}^2 + \frac{1}{2}|\mathcal{G}(u^\dagger) - \mathcal{G}(v)|_{\Gamma(v)}^2\\
&\hspace{1cm}+\frac{5}{\lambda_{\min}^*n^2}|\eta^m_n\mathcal{G}(u^\dagger)+\eta^a_n|_Y^2.
\end{align*}
The Lipschitz property of $\mathcal{G}$ on $X$ tells us that $|\mathcal{G}(u^\dagger) - \mathcal{G}(v)|_{\Gamma(v)} \leq C\eps^2$, hence taking expectations above we deduce that
\[
\mathbb{E}|\mathcal{G}(u^\dagger)-\mathcal{G}(u_n)|_{\Gamma(u_n)}^2 \leq C\eps^2 + \frac{C_\eps}{n^2}
\]
and so
\[
\limsup_{n\rightarrow\infty}\mathbb{E}|\mathcal{G}(u^\dagger)-\mathcal{G}(u_n)|_{\Gamma(u_n)}^2 \leq C\eps^2.
\]
Since the $\liminf$ is non-negative and $\eps > 0$ is arbitrary, the expectation must converge to zero. We deduce that $\Gamma(u_n)^{-1/2}(\mathcal{G}(u^\dagger)-\mathcal{G}(u_n))\rightarrow 0$ along a subsequence almost surely, and so by the same argument as in the previous proof, $\mathcal{G}(u_n)\rightarrow\mathcal{G}(u^\dagger)$ along a subsequence almost surely.
\end{proof}

\subsection{Large data limit}
We provide some insight into the large data limit of MAP estimators. Again we assume that there exists an true state $u^\dagger \in X$ so that the data $y$ arises from the equation
\[
y = (\mathbf{1}+\eta^m)\mathcal{G}(u^\dagger) + \eta^a.
\]
Suppose we have a sequence of independent observations $y = (y_k)_{k=1}^n$, $y_k \in \mathbb{R}^J$, so that
\begin{align*}
y_k = (\mathbf{1}+\eta^m_k)\mathcal{G}(u^\dagger) + \eta^a_k,\;\;\;&k=1,\ldots,n\\
&\eta^m_k\sim N(0,\Gamma^m)\text{ iid},\\
&\eta^a_k \sim N(0,\Gamma^a)\text{ iid}.
\end{align*}
This can be cast in the previous setting by concatenating the observations, so $Y = \mathbb{R}^{nJ}$. We study the limit of the resulting posterior as $n\rightarrow\infty$. Using the same prior on $u$ as before, since the observations are independent, the posterior distribution is given by
\[
\frac{\dee \mu^{y_1,\ldots,y_n}}{\dee \mu_0}(u) \propto \exp\left(-\frac{1}{2}\sum_{k=1}^n |\mathcal{G}(u)-y_k|_{\Gamma(u)}^2 - \frac{n}{2}\log\det\Gamma(u)\right),
\]
where $\Gamma(u)$ is as defined previously. By \cref{thm:map_mixed}, MAP estimators of the posterior $\mu^{y_1,\ldots,y_n}$ are minimizers of
\[
I_n(u) := \frac{1}{2}\|u\|_E^2 + \frac{1}{2}\sum_{k=1}^n|\mathcal{G}(u)-y_k|_{\Gamma(u)}^2 + \frac{n}{2}\log\det\Gamma(u).
\]
As before, we rescale $I_n$ by $2/n$, preserving its minimizers. We look at the limit of $2I_n(u)/n$ for large $n$ using the law of large numbers.

\begin{proposition}
\label{limitfunc}
Define the functionals $J_n:E\rightarrow\mathbb{R}$ by
\[
J_n(u) := \frac{2}{n}I_n(u) = \frac{1}{n}\|u\|_E^2 + \frac{1}{n}\sum_{k=1}^n |\mathcal{G}(u)-y_k|^2_{\Gamma(u)} + \log\det\Gamma(u).
\]
For all $u \in E$, we have $\mathbb{P}$-almost surely
\[
\lim_{n\rightarrow\infty} J_n(u) = |\mathcal{G}(u^\dagger) - \mathcal{G}(u)|^2_{\Gamma(u)} + \tr(\Gamma(u^\dagger)\Gamma(u)^{-1}) + \log\det\Gamma(u) =: \mathcal{J}(u).
\]
\end{proposition}

\begin{proof}
We use the strong law of large numbers. Fix $u \in E$ and define the sequence of iid random variables $(X_k)$ by
\begin{align*}
X_k &= |\mathcal{G}(u) - y_k|_{\Gamma(u)}^2 = |\mathcal{G}(u^\dagger) - \mathcal{G}(u) + \eta^m_k\mathcal{G}(u^\dagger) + \eta^a_k|_{\Gamma(u)}^2.
\end{align*}
Using the independence of $\eta^m \sim N(0,\Gamma^m)$ and $\eta^a \sim N(0,\Gamma^a)$, we see that these $X_k$ have mean given by
\begin{align*}
\mathbb{E}X_k &= \mathbb{E}|\mathcal{G}(u^\dagger) - \mathcal{G}(u)|^2_{\Gamma(u)} + \mathbb{E}|\eta_k^m\mathcal{G}(u^\dagger) + \eta_k^a|_{\Gamma(u)}^2\\
&\hspace{1cm}+ \mathbb{E}\langle\mathcal{G}(u^\dagger) - \mathcal{G}(u),\eta_k^m\mathcal{G}(u^\dagger) + \eta_k^a\rangle_{\Gamma(u)}\\
&= |\mathcal{G}(u^\dagger) - \mathcal{G}(u)|^2_{\Gamma(u)} + \mathbb{E}\sum_{i,j=1}^J \Gamma(u)^{-1}_{ij}(\eta^m_{k,i}\mathcal{G}(u^\dagger) + \eta^a_{k,i})(\eta^m_{k,j}\mathcal{G}(u^\dagger) + \eta^a_{k,j})\\
&\hspace{1cm}+ \mathbb{E}\sum_{i,j=1}^J \Gamma(u)^{-1}_{ij} (\mathcal{G}_i(u^\dagger)-\mathcal{G}_i(u))(\eta^m_{k,j}\mathcal{G}(u^\dagger) + \eta^a_{k,j})\\
&= |\mathcal{G}(u^\dagger) - \mathcal{G}(u)|^2_{\Gamma(u)} + \sum_{i,j=1}^J\Gamma(u)^{-1}_{ij}(\Gamma^a_{ij} + \mathcal{G}_j(u^\dagger)\mathcal{G}_j(u^\dagger)\Gamma^m_{ij})\\
&= |\mathcal{G}(u^\dagger) - \mathcal{G}(u)|^2_{\Gamma(u)} + \tr(\Gamma(u^\dagger)\Gamma(u)^{-1}).
\end{align*}
Using the bound $X_k \leq (|\mathcal{G}(u)|_Y + |y_k|_Y)^2$, and the fact $y_k \sim N(\mathcal{G}(u^\dagger),\Gamma(u^\dagger))$ coupled with the property that Gaussians have moments of all orders, we deduce that $X_k$ has finite variance. These conditions are sufficient for the strong law of large numbers to hold for $X_k$, and so we deduce that $\mathbb{P}$-almost surely,
\begin{align*}
\lim_{n\rightarrow\infty} J_n(u) &= \lim_{n\rightarrow\infty}\frac{1}{n}\|u\|^2_E + \lim_{n\rightarrow\infty}\frac{1}{n}\sum_{k=1}^n X_k + \log\det\Gamma(u)\\
&= |\mathcal{G}(u^\dagger) - \mathcal{G}(u)|^2_{\Gamma(u)} + \tr(\Gamma(u^\dagger)\Gamma(u)^{-1}) + \log\det\Gamma(u).
\end{align*}
\end{proof}

The limiting functional $\mathcal{J}$ is nice in that we can completely characterize its global minimizers:

\begin{proposition}
\label{limitmin}
The global minimizers of $\mathcal{J}$ are precisely the $u^* \in E$ for which $\mathcal{G}(u^*) = \mathcal{G}(u^\dagger)$.
\end{proposition}

\begin{proof}
Let $u^* \in E$ be such that $\mathcal{G}(u^*) = \mathcal{G}(u^\dagger)$, and let $u \in E$ be any $u$ such that $\mathcal{G}(u) \neq \mathcal{G}(u^\dagger)$. Then we have $\Gamma(u^*) = \Gamma(u^\dagger)$, and so
\begin{align*}
\mathcal{J}(u) - \mathcal{J}(u^*) &= |\mathcal{G}(u^\dagger) - \mathcal{G}(u)|_{\Gamma(u)}^2 + \tr(\Gamma(u^\dagger)\Gamma(u)^{-1}) + \log\det\Gamma(u)\\
&\hspace{1cm} -|\mathcal{G}(u^\dagger) - \mathcal{G}(u^*)|_{\Gamma(u^*)} - \tr(\Gamma(u^\dagger)\Gamma(u^*)^{-1}) - \log\det\Gamma(u^*)\\
&= |\mathcal{G}(u^\dagger)-\mathcal{G}(u)|^2_{\Gamma(u)} + \tr(\Gamma(u^\dagger)\Gamma(u)^{-1}) + \log\det\Gamma(u)\\
&\hspace{1cm}- J - \log\det\Gamma(u^\dagger)\\
&=|\mathcal{G}(u^\dagger)-\mathcal{G}(u)|^2_{\Gamma(u)} + \tr(\Gamma(u^\dagger)\Gamma(u)^{-1}) - \log\det(\Gamma(u^\dagger)\Gamma(u)^{-1}) - J\\
&=|\mathcal{G}(u^\dagger)-\mathcal{G}(u)|^2_{\Gamma(u)} + \tr(A(u)) - \log\det A(u) - J
\end{align*}
where $A(u) = \Gamma(u^\dagger)\Gamma(u)^{-1}$ is a positive definite matrix. Let $(\lambda_i(u))_{i=1}^J$ denote the (positive) eigenvalues of $A(u)$. Expanding the trace and determinant in terms of these yields
\[
\mathcal{J}(u) - \mathcal{J}(u^*) = |\mathcal{G}(u^\dagger)-\mathcal{G}(u)|^2_{\Gamma(u)} +  \sum_{i=1}^J \left(\lambda_i(u) - \log\lambda_i(u)\right) - J.
\]
It can be easily checked (e.g. with one-dimensional calculus) that for any positive $x$ we have $x-\log{x} \geq 1$, and so the right-hand side above is non-negative. Moreover, since $\mathcal{G}(u) \neq \mathcal{G}(u^\dagger)$ and $\Gamma(u)$ is positive definite, the first term must be strictly positive and the result follows.
\end{proof}

\begin{remark}
Note that in the small-noise limit, the logarithmic term disappears in the limiting objective functional, and so for very small noise its inclusion is not necessarily important for consistent estimates. However, if it is neglected in the large sample size limit, the limiting functional will not possess the correct minimizers, and the MAP estimators will not be consistent. This is distinct from the purely additive noise case analysed in \cite{DLSV13}, in which the limiting functionals were the same for both small noise and large sample limits.
\end{remark}

\section{Conclusions and future work}
We have studied the role of multiplicative noise in Bayesian inverse problems. We considered two general models: purely multiplicative noise and a mixed multiplicative/additive model. The latter is shown to be more robust in that weaker assumptions are required on the forward model in order to establish continuity of the posterior in the Hellinger metric with respect to perturbations in the data. In the case of the Gaussian mixed noise model we have shown consistency of MAP estimators in the small noise limit, and provided motivation for the analogous result in the large data limit.

The data model may be generalized further to
\[
y = H(u)\eta,\quad \eta \sim N(z_0,\Gamma),
\]
where $H:X\to\mathcal{L}(Z;Y)$ takes values as general linear operators, rather than a multiplication operators. The additive, multiplicative and mixed noise models are all special cases of this model. In general we may allow the data space $Y$ to be infinite-dimensional, which essentially allows for certain stochastic forward maps, instead of the deterministic ones considered previously. For example, $H(u)$ could be the solution operator $\eta\mapsto p$ for the SPDE
\[
-\nabla\cdot(e^u\nabla p) = \eta
\]
subject to appropriate boundary conditions. The likelihood takes a similar form to the multiplicative case, with
\[
y|u \sim \mathbb{Q}_u := N(H(u)z_0,H(u)\Gamma H(u)^*).
\]
We expect that, subject to appropriate additional assumptions to ensure existence of a dominating measure $\mathbb{Q}_0$ with $\mathbb{Q}_u \ll \mathbb{Q}_0$ for all $u \in X$, similar results should be able to be attained as in this article.

\begin{appendix}
\section{Additional proofs}

\begin{proof}[Proof of \cref{prop:mixed_density}]
Let $f:Y\rightarrow\mathbb{R}$ be bounded measurable. Then
\begin{align*}
\mathbb{E}f(y|u) &= \int_{Y}\int_{Y} f(A(u)w+ z)\rho_a(z)\rho_m(w)\,\dee w\dee z.
\end{align*}
Performing the change of variables
\[
\begin{pmatrix}
x\\
y
\end{pmatrix}
= 
\begin{pmatrix}
I & 0\\
A(u) & I
\end{pmatrix}
\begin{pmatrix}
w\\
z
\end{pmatrix}
\]
we deduce that
\begin{align*}
\mathbb{E}f(y|u) &= \int_{Y}f(y)\left(\int_{Y}\rho_a(y - A(u)x)\rho_m(x)\,\dee x\right)\dee y.
\end{align*}
\end{proof}

\begin{proof}[Proof of \cref{prop:gauss_assump}]
\cref{assump:mixed}(ii) can be easily checked, giving $q_1 = 2$ and $q_2 = 1$, and  \cref{assump:mixed}(iii) follows since Gaussian random variables have all polynomial moments. The proof that \cref{assump:mixed}(iv) holds is more involved. Let $A$ be a diagonal matrix. The probability measure $\nu^{y,A}$ is given by
\begin{align*}
\nu^{y,A}(\dee x) &\propto \rho_a(y-Ax)\rho_m(x)\,\dee x\\
&\propto \exp\left(-\frac{1}{2}\left(|y-Ax|_{\Gamma^a}^2 + |x-\mathbf{1}|^2_{\Gamma^m}\right)\right)\,\dee x.
\end{align*}
This is a Gaussian measure on $Y$. Expanding and completing the square, we see that its mean $m$ and covariance $\Sigma$ are given by
\begin{align*}
\Sigma &= (\Gamma^{m,-1} + A\Gamma^{a,-1}A)^{-1},\\
m &= \Sigma (A\Gamma^{a,-1}y + \Gamma^{m,-1}\mathbf{1}).
\end{align*}
Let $X \sim \nu^{y,A}$. We wish to bound $\mathbb{E}^{\nu^{y,A}}|X|_Y$. Using equivalence of finite dimensional norms, we have
\[
\mathbb{E}^{\nu^{y,A}}|X|_Y \leq C\sum_{j=1}^J\mathbb{E}|X_i|.
\]
Each of the marginal distributions are given by $X_i \sim N(m_i,\Sigma_{ii})$. We use the Sherman-Woodbury formula to deduce
\begin{align*}
\Sigma &= (\Gamma^{m,-1} + A\Gamma^{a,-1}A)^{-1}\\
&= \Gamma^m - \Gamma^mA(\Gamma^a + A\Gamma^m A)^{-1}A\Gamma^m\\
&= \Gamma^m - \Gamma^mA(I + \Gamma^{a,-1}A\Gamma^m A)^{-1}\Gamma^{a,-1}A\Gamma^m.
\end{align*}
Now since $A$ is diagonal, so is its product with any other matrix, and so
\begin{align*}
\Sigma_{ii} &= \Gamma^m_{ii} - \Gamma^{m}_{ii}A_{ii}[(I+A\Gamma^{a,-1}\Gamma^m A)^{-1}]_{ii}\Gamma^{a,-1}_{ii}A_{ii}\Gamma^m_{ii}\\
&= \Gamma^m_{ii} - \frac{A_{ii}^2(\Gamma^m_{ii})^2\Gamma_{ii}^{a,-1}}{1 + A_{ii}^2\Gamma^m_{ii}\Gamma_{ii}^{a,-1}}\\
&= \frac{\Gamma^m_{ii}}{1 + A_{ii}^2\Gamma_{m,ii}\Gamma_{ii}^{a,-1}} \leq \Gamma^{m}_{ii}.
\end{align*}
The mean is then given by $m_i = \Sigma_{ii}A_{ii}\Gamma^{a,-1}_{ii}y_i + \Sigma_{ii}\Gamma^{m,-1}_{ii}$.
Now standard results for one-dimensional Gaussian random variables tell us that the non-central first absolute moment $\mathbb{E}|X_i|$ is given by
\[
\mathbb{E}|X_i| = \sqrt{\frac{2\Sigma_{ii}}{\pi}} {}_{1}F_1\left(-\frac{1}{2},\frac{1}{2},-\frac{1}{2}\frac{m_i^2}{\Sigma_{ii}}\right),
\]
where ${}_1F_1$ is the Kummer confluent hypergeometric function, defined by
\[
{}_1F_1(a,b,x) = \sum_{k=0}^\infty \frac{(a)_k}{(b)_k}\frac{x^k}{k!}.
\]
By comparing power series, it can be seen that we have for $x \geq 0$
\begin{align*}
{}_1F_1(-1/2,1/2,-x) = e^{-x} + \mathrm{erf}(\sqrt{x})\sqrt{\pi x} \leq 1 + \sqrt{\pi x} \leq C(1 + x).
\end{align*}
We may therefore bound the moment
\begin{align*}
\mathbb{E}|X_i| &\leq C\sqrt{\frac{2\Sigma_{ii}}{\pi}}\left(1 + \frac{1}{2}\frac{m_i^2}{\Sigma_{ii}}\right)\\
&= C\sqrt{\frac{2\Sigma_{ii}}{\pi}}\left(1 + \frac{1}{2}\Sigma_{ii}(A_{ii}\Gamma^{a,-1}_{ii}y_i + \Gamma^{m,-1}_{ii})^2\right)\\
&\leq C\sqrt{\frac{2\Gamma^{m}_{ii}}{\pi}}\left(1+\frac{1}{2}\Gamma^{m}_{ii}(A_{ii}\Gamma^{a,-1}_{ii}y_i + \Gamma^{m,-1}_{ii})^2\right)\\
&\leq C(1+ y_i^2A_{ii}^2)\\
&\leq C(1+y_i^4 + A_{ii}^4).
\end{align*}
It follows from the equivalence of finite-dimensional norms that
\[
\mathbb{E}^{\nu^{y,A}}|X|_Y \leq C(1 + |y|_Y^4 + \|A\|_{\mathrm{op}}^4).
\]
\end{proof}

\end{appendix}

\section*{Acknowledgements}
This work was supported in part by the U.S. Department of Energy Office of Science, Advanced Scientific Computing Research (ASCR), Scientific Discovery through Advanced Computing (SciDAC) program. The author would like to thank Georg Stadler and Andrew Stuart for helpful discussions.

\bibliographystyle{siamplain}
\bibliography{references}   

\end{document}